\documentclass[12pt, twoside, leqno]{article}

\usepackage{amsmath,amsthm}
\usepackage{amssymb}
\usepackage{xcolor}

\usepackage{enumitem}

\usepackage{graphicx}

\usepackage[T1]{fontenc}

\newtheorem{theorem}{Theorem}[section]
\newtheorem{corollary}[theorem]{Corollary}
\newtheorem{lemma}[theorem]{Lemma}
\newtheorem{proposition}[theorem]{Proposition}

\theoremstyle{definition}
\newtheorem{definition}[theorem]{Definition}

\newtheorem{thmx}{Theorem}

\numberwithin{equation}{section}

\begin{document}


\baselineskip=17pt


\title{Borel complexity of the set of typical numbers}

\author{Jakub Tomaszewski\\
Address: Faculty of Applied Mathematics, AGH University\\
Mickiewicza 30, 30-059 Kraków\\
E-mail: tomaszew@agh.edu.pl}

\date{}

\maketitle


\renewcommand{\thefootnote}{}



\renewcommand{\thefootnote}{\arabic{footnote}}
\setcounter{footnote}{0}

\begin{abstract}
In the present note we study the interrelations between the sets of so-called typical numbers and numbers that are normal in base two. Employing results by Nakai and Shiokawa, we exhibit examples of numbers that belong to one set but do not belong to the other and vice versa. Moreover, we demonstrate that the set of typical numbers is $\Pi_3^0$ in the Borel hierarchy, i.e., it can be expressed as the union of countably many $F_{\sigma}\textnormal{-sets.}$ Using the result by Ki and Linton that asserts the same for normal numbers, we examine the Borel complexity of the set of typical numbers that are not normal, proving that it is a $D_2(\Pi_3^0)$ set.
\end{abstract}

\maketitle

\section{Introduction}
In 1970, Paul Erdős and Alfred Rényi have shown in \cite{Erdos} a "new law of large numbers". They were examining the fair game of independent repetitions, where at each round the outcomes are either $1$ or $-1$. A special case of the theorem they proved asserts that
\begin{equation*}
    P\left(\lim_{N\to\infty} \max_{0\leq n\leq N-[\log_2 N]}\frac{S_{n+[\log_2 N]}-S_n}{[\log_2 N]}=1\right)=1,
\end{equation*}
where $[x]$ is the integer part of $x$ and $S_k$ is a total gain up to $k\textnormal{-th}$ round. This means that almost surely the maximal amount of consecutive successful rounds after $N$ rounds played tends to $[\log_2 N]$ as the game progresses.
\par The course of that game can be identified with binary sequences, relating this result with real numbers on $[0,1]$. By the above law, we obtain that for almost every $x\in [0,1]$ the following limit exists: $$\mathcal{L}_n(x)\xrightarrow[n\to\infty]{}1,$$ where $\mathcal{L}_n(x)=\frac{L_n(x)}{\log_2 n}$ and $L_n(x)$ is a "run-length" function, returning the maximal length of a sequence of consecutive ones up to the $n\textnormal{-th}$ position in the binary expansion of $x$. Those $x\in [0,1]$, whose binary expansions do not satisfy the above convergence we call exceptional in the sense of Erdős-Rényi law. There has been a number of studies on exceptional sets, e.g. by Sun and Xu in  \cite{Sun}. In particular, they examined the Hausdorff dimension of specific sets 
\begin{equation*}
    E_{\alpha,\beta} = \{x\in[0,1]\colon \liminf\mathcal{L}_n(x)=\alpha,\; \limsup\mathcal{L}_n(x)=\beta\}
\end{equation*}
for two numbers $\alpha,\, \beta\in[0,\infty]$ satisfying $\alpha\leq\beta$.
\par In this paper we will examine the set of all exceptional numbers, as well as its complement - the set of typical numbers, for which $\mathcal{L}_n$ indeed converges to 1. It turns out that the set of typical numbers is Borel and thus we can investigate its Borel complexity. This classification gives us a way to express how complex is the construction of a set. In simple words, we start with open or closed sets and look where we exceed these topological families by the operations of countable union and countable intersection. Every Borel set can be expressed using countably many of the above mentioned operations. Thus we associate a notion of level to the amount of operations used during the process of construction. It is common for a set to be no higher than at level 3 in this hierarchy, meaning combining two exceeding operations. We will show that the set of typical numbers is no different, placing it exactly at level 3.
\begin{thmx}\label{theoremA}
    The set of typical numbers is Borel and $\Pi^0_3([0,1])\text{-complete}$.
\end{thmx}
\par However, the set of typical numbers which are not normal happens to be above level 3. Normal numbers in $[0,1]$ are those, whose binary expansions satisfy the following property: for a given finite sequence of length $n$, its frequency of appearance is exactly $\frac{1}{2^n}$, meaning that the occurrence of every block of a given length is equally probable. The set of normal numbers is also of full measure on $[0,1]$. Some might say that relating normal numbers with typical numbers is natural. Since the chance to observe a sequence of $n$ consecutive ones in the initial part of length $2^n$ of the infinite binary expansion tends to 1 as $n$ tends to infinity, every normal number could be expected to be typical. How natural this correlation may be, it is non-trivial and we will construct examples of a typical abnormal number and a normal exceptional number. We will prove the following theorem:
\begin{thmx}\label{theoremB}
    The sets $\textit{TpcN}\setminus\textit{Normal},\; \textit{Normal}\setminus\textit{TpcN}$ are nonempty and they are $D_2(\Pi^0_3)$-complete.
\end{thmx}
The above result has a big consequence on the interrelations between typicality and normality, in view of the comment made in \cite{Bill} on the difference hierarchy. Namely, if for two $\Gamma$-complete sets $A,B$ their difference set $A\setminus B$ is $D_2(\Gamma)$-complete, then the difference has maximum logical complexity. Thus the correlation between being typical and being normal is not only non-trivial, but in fact any additional condition we would have to consider in order to observe such correlation has to be at least as logically complex as these notions are.

\section{Preliminaries}
Let $X=\{0,1\}^\mathbb{N}$ be the space of infinite binary sequences with the product topology induced from the discrete topology on $\{0,1\}$. A compatible metric is defined for $x,y\in X$ as follows:
\begin{equation*}
    d(x,y):=\begin{cases}
             0, \;\text{if} \:x=y, \\
             \frac{1}{2^{k-1}},\; \text{where} \:k=\min\{j\in\mathbb{N}\colon x_j\neq y_j\}, \;\text{if otherwise.}
        \end{cases}
\end{equation*}
Now consider a function $L_n(x)\colon\{0,1\}^\mathbb{N}\to\mathbb{N}\cup \{0\}$ defined in the following way:
\begin{equation*}
    L_n(x):=\max\{j\colon x_{i+1}=...=x_{i+j}=1\; \text{for some}\; 0\leq i\leq n-j\}.
\end{equation*}
The function $L_n(x)$ returns the maximal length of a sequence of consecutive ones appearing up to the $n$-th position in the sequence $x$.

\subsection{Typical numbers}
It is clear that we can associate a number $y\in[0,1]$ with a sequence $x\in\{0,1\}^\mathbb{N}$ using the following formula:
\begin{equation}\label{summation}
    y=\sum_{n=1}^\infty\frac{x_n}{2^n}.
\end{equation}
From now on, we will identify a number $y\in [0,1]$ with its corresponding binary expansion $x\in\{0,1\}^\mathbb{N}$. Although the expansion need not be unique, the choice of it is not essential for our considerations. It is possible to have at most two distinct binary expansions, and there are countably many numbers having this property. When it occurs, we have even more: first one must be finite, so the last 1 is represented by a 0 followed an infinite sequence of ones in the second expansion. Finite case does not satisfy any of the conditions for normality or typicality, since both of them require limit convergence. In the infinite one, frequencies are clearly biased simply for 0 and 1, since the sequence of consecutive ones becomes almost as long as the whole sequence when $n\to\infty$. Thus a number which has more than one expansion, cannot be typical, nor it can be normal.
\par By Erdős–Rényi Theorem (\cite{Erdos}) we know that for almost all $x\in [0,1]$
\begin{equation}\label{condition_above}
    \lim_{n\to\infty}\mathcal{L}_n(x)=\lim_{n\to\infty}\frac{L_n(x)}{\log_2n}=1,
\end{equation}
with respect to the standard Lebesgue measure on $[0,1]$.
We call a number typical when its binary expansion satisfies the condition (\ref{condition_above}), and exceptional when it doesn't. Let \textit{TpcN} denote the set of typical numbers and \textit{ExcN} denote the set of exceptional numbers.
\subsection{Normal numbers}
Let us briefly remind the definition of the set of normal numbers in $[0,1]$. We call a number normal in base 2 if each finite block of digits of length $m$ appears with frequency $\frac{1}{2^m}$ in the binary expansion of that number. Here, for a given finite block $\omega$ we define its frequency as $\lim_{n\to\infty} \frac{N(x, \omega, n)}{n}$, where $N(x, \omega, n)$ represents the amount of appearances of block $\omega$ in first $n$ positions of a binary expansion of a number $x\in[0,1]$. We denote the set of normal numbers $Normal$ and define $\textit{Abnormal}=[0,1]\setminus \textit{Normal}$. The set of normal numbers is Borel. To add to that, it is $\Pi^0_3$-complete and of full measure with respect to the Lebesgue measure on $[0,1]$, as stated in \cite{Dominik}.
\par One of the well-known normal numbers is the Champernowne constant (by \cite{Champernowne}, \cite{Nakai}). Champernowne constant is a number defined by its expansion, which in base 2 is constructed by concatenating binary notations of consecutive natural numbers:
\begin{equation*}
    \mathcal{C}_{2}=0.\;1\;10\;11\;100\;101\;110\;111\;...
\end{equation*}
We have the following theorem which gives a way to transform Champernowne constant into many other normal numbers via real polynomials:
\begin{theorem}[\cite{Nakai}]\label{Nakai}
Let $w(x)=a_nx^n+...+a_1x+a_0$ be a real polynomial that satisfies $w(x)>0$ for all $x>0$ and there exists $i\in\{1,...,n\}$ such that $a_i\neq 0$. The number
\begin{equation*}
    0.\;[w(1)]_r\;[w(2)]_r\;[w(3)]_r\;[w(4)]_r\;...,
\end{equation*}
where $[w(x)]_r$ is the integer part of $w(x)$ in base $r$, is normal.
\end{theorem}

\subsection{Classifying set constructions}\label{classification}
Since the family of open sets and closed sets in $X$ is not closed under both operations of countable unions and countable intersections, we need to consider a larger class of sets, Borel $\sigma$-algebra on $X$, denoted by $\mathcal{B}(X)$. On that class we are able to determine the complexity of construction of a given set.  We achieve it with Borel hierarchy (further reading can be found in \cite{Dominik}, \cite{Bill}, \cite{Kechris}, \cite{Miller}).

\subsubsection{Borel hierarchy}\label{borel_hierarchy_introduced}
Let $\Sigma_1^0(X)$ denote the collection of open sets and $\Pi^0_1(X)=\{X\setminus A\colon A\in \Sigma_1^0(X)\}$ be the collection of closed sets. For a countable ordinal $m$ let $\Sigma^0_m(X)$ be collection of countable unions $A=\bigcup_{n\in\mathbb{N}}X\setminus B_n$ such that $B_n\in\Pi_{\beta_n}^0(X)$ for $\beta_n<m$ and similarly $\Pi^0_m(X)=\{X\setminus A\colon A\in \Sigma_m^0(X)\}$. We also define $\Delta^0_n(X)=\Sigma^0_n(X)\cap\Pi^0_n(X)$. These pointclasses define the Borel hierarchy on Borel $\sigma$-algebra of sets for a topological space $X$. We use the abbreviated notation $\Pi_n^0, \Sigma_n^0, \Delta_n^0$, when it is clear on which space the hierarchy is defined.

We say that set $A$ is $\Sigma^0_n$-hard if $A\not\in\Pi^0_n$. The notion of hardness introduces a lower bound for the set construction, showing that $A$ is no simpler than a $\Sigma^0_n$ set. We also say that set $A$ is $\Sigma^0_n$-complete if $A\in\Sigma^0_n\setminus\Pi^0_n$. Note that this is the same to say that $A$ is $\Sigma^0_n$-complete if $A$ is $\Sigma^0_n$-hard and $A\in\Sigma^0_n$. Likewise, the set $A$ is $\Pi^0_n$-hard if $A\not\in\Sigma^0_n$ and $\Pi^0_n$-complete if $A\in\Pi^0_n\setminus\Sigma^0_n$.

\begin{theorem}[\cite{Miller}]\label{union_intersection}
For every countable ordinal $m$, the family $\Sigma^0_m$ is closed under countable unions and finite intersections, $\Pi^0_m$ is closed under countable intersections and finite unions, $\Delta^0_m$ is closed under finite unions and finite intersections.
\end{theorem}
\begin{theorem}[\cite{Miller}]\label{inclusion}
We have the following class inclusions:
\begin{enumerate}
    \item $\Pi_{n}^0\subset\Sigma_{n+1}^0$,
    \item $\Sigma_n^0\subset\Pi_{n+1}^0$,
    \item if $\Pi_1^0\subset\Pi_2^0$, then
    \begin{itemize}
        \item $\Pi_n^0\subset\Pi_{n+1}^0$,
        \item $\Sigma_n^0\subset\Sigma_{n+1}^0$, and consequently
        \item $\Pi_n^0\cup\Sigma_n^0\subset\Delta_{n+1}^0$.
    \end{itemize}
\end{enumerate}
\end{theorem}
There are not many natural examples of sets that have level 4 or higher the hierarchy defined above. Vast majority of sets appear in some previous class, as mentioned in \cite{Dominik}. Observe that, as a consequence of Theorem \ref{inclusion}, every set belongs to infinite amount of classes, but appears in only one. We refer to "appearance" or "being" to express class completeness.
\subsubsection{Difference Hierarchy}
Let $(A_\nu)_{\nu<\theta}$ be an increasing sequence of subsets of $X$ for a countable non-zero ordinal $\theta$. Observe that each ordinal $\theta$ can be written as $\lambda+n$ where $\lambda$ is the limit or 0 and and $n<\omega$. Here we consider the parity of $\theta$ as the parity of $n$. Define the set $D_{\theta}((A_\nu)_{\nu<\theta})\subset X$ as a set of all those $x\in X$ such that they belong to $\bigcup_{\nu<\theta} A_{\nu}$ and the smallest $\nu$ such that $x\in A_{\nu}$ has the opposite parity to the one of $\theta$. Thus we obtain a sequence
\begin{equation*}
    \begin{split}
        &D_1(A_0)=A_0, \;D_2(A_1, A_0)=A_1\setminus A_0,\\
        &D_3(A_2, A_1, A_0)=(A_2\setminus A_1)\cup A_0,...
    \end{split}
\end{equation*}
Further details can be found in (\cite{Kechris}, 22.E).
\begin{theorem}[\cite{Kechris}, 22.27]\label{pomocy}
For a polish space \textit{X} let $$D_{\theta}(\Sigma_n^0)(X)=\{D_{\theta}((A_{\nu})_{\nu<\theta})\colon A_{\nu}\in\Sigma_n^0\}.$$ Then $\Delta_{n+1}^0=\bigcup_{1\leq\theta<\omega_1}D_{\theta}(\Sigma_n^0)(X)$.
\end{theorem}
By the above theorem and the fact that $\Sigma_n^0\subset D_2(\Sigma_n^0)$ one can see that $D_n$ classes live between levels in the Borel hierarchy. Similarly as for the $\Sigma_n^0$ sets, we can consider the above construction for $\Pi_n^0$ sets. Especially, we will focus on the class $D_2(\Pi_n^0)$, which consists of sets $C=A\setminus B$, where $A,B \in\Pi_n^0$. For $D_n$ classes we can examine the same notions of hardness and completeness as for $\Pi_n, \Sigma_n$ classes. It is reasonable to do so in view of Theorem \ref{pomocy}, since no set is $\Delta_n$-complete for $n\geq 2$ (see \cite[22.28]{Kechris}).

\section{Main results}
\subsection{Proof of Theorem \ref{theoremA}}
Before proceeding with the proof, let us discuss the techniques we will use and the steps we will follow. In this section, we will focus on the set of typical numbers, which we denote by $\textit{TpcN}$. We shall identify its class in the Borel hierarchy in two steps. Firstly we will rewrite its definition as a composition of unions and intersections of basic sets. Then we will construct a continuous map $f\colon\mathbb{N}^\mathbb{N}\to \{0,1\}^\mathbb{N}$ which will satisfy the condition $f^{-1}(\textit{TpcN})=P_3$, where $P_3=\{(a_n)_{n=1}^\infty\subset\mathbb{N}\colon\;\liminf_{n\to\infty} a_n=\infty\}$ is the set of natural sequences diverging to infinity, a well-known example of a $\Pi_3^0$-complete set (as stated in \cite{Dominik}). \par This approach follows a more general method of determining the Borel complexity. It consists of finding some construction of the set of interest and then defining a continuous map onto some other set, whose complexity is known. The former technique is simply determining an upper bound in the hierarchy, showing that the optimal set construction is no more complex than the one we wrote down. The latter one is known as the Wadge reduction:
\begin{theorem}[\cite{Dominik}]\label{Wadge}
Let $X,Y$ be Polish spaces and $f\colon X\to Y$ be a continuous function. Let $A\in\mathcal{B}(X), B\in\mathcal{B}(Y)$. If $A=f^{-1}(B)$, then they belong to the same Borel hierarchy classes in their $\sigma$-algebras, respectively.
\end{theorem}

Let $\Gamma$ be any class from the Borel hierarchy. Observe that whenever one constructs a reduction from a $\Gamma$-complete set $A$ to some set $B$, then by the above theorem it immediately follows that $B$ is $\Gamma$-hard. Thus we mainly use the Wadge reduction for establishing a lower bound for given set construction.

\subsubsection{$\textit{TpcN}\in\Pi^0_3$}\label{upper_bound_tpcn}
Let us rewrite the condition defining a typical number in a different way. As we have that $\mathcal{L}_n$ converges to 1, then the set of typical numbers is a set of all $x\in [0,1]$ such that
\begin{equation*}
    \forall m\in\mathbb{N}\;\exists n_0\in\mathbb{N}\;\forall n>n_0 \;\frac{L_n(x)}{\log_2n}\in(1-\frac{1}{m}, 1+\frac{1}{m}).
\end{equation*}
Thus the set $TpcN$ is constructed in the following way:
\begin{equation*}
    \bigcap_{m\in\mathbb{N}}\;\bigcup_{n_0\in\mathbb{N}}\;\bigcap_{n>n_0} \;\{x\in [0,1]\colon \frac{L_n(x)}{\log_2n}\in(1-\frac{1}{m}, 1+\frac{1}{m})\}.
\end{equation*}
Fix some $n>0$ for further look at the characteristics of the set 
\begin{equation}\label{condition}
    A=\{x\in [0,1]\colon \frac{L_n(x)}{\log_2n}\in(1-\frac{1}{m}, 1+\frac{1}{m})\},
\end{equation}

\begin{lemma}\label{jest_Pi_3}
The set A is closed.
\end{lemma}

\begin{proof}
It is clear that as $n$ is fixed, then $L_n(x)\in \{0,1,...,n\}$ and $\log_2 n$ are fixed. Thus we have that there is a finite number of possibilities for initial sequences in binary expansions for which $x\in \{0,1\}^\mathbb{N}$ will satisfy the condition (\textnormal{\ref{condition}}). To find those sequences we need to determine a set $\Omega_n$ of blocks $\omega$ of length $n$ such that every sequence $x\in \{0,1\}^\mathbb{N}$ satisfying $x_i=\omega_i$ for $i=1,...,n$ is appropriate for the condition above. As $n$ is finite, then the number of admissible blocks is no greater than $\sum_{i=0}^n \binom{n}{i}$ and therefore finite. For given $\omega \in \Omega_n$ let $B(\omega, \frac{1}{2^{n}})$ denote a set of all sequences starting with block $\omega$, which we will later call a cylinder and denote by $[\omega]$ or $[\omega_1...\omega_n]$. As we consider $\{0,1\}^\mathbb{N}$ with product topology, then by continuous projections onto i-th coordinate $\Pi_i$ we obtain that
\begin{equation*}
    [\omega_1...\omega_n]=\bigcap_{i=1}^n \Pi^{-1}_i(\{\omega_i\}),
\end{equation*}
where $\omega_i\in \{0,1\}$ for $i=1,...,n$. We conclude that the set $[\omega_1...\omega_n]$ is closed, since the intersection of any family of closed sets is closed, and then the set of all admissible sequences $\bigcup_{\omega \in\Omega_n}[\omega]$ is closed. The map defined by (\ref{summation}) is continuous and therefore closed by Closed Map Lemma. Thus $A$ is closed.
\end{proof}
  By the above lemma one can see that $\textit{TpcN}$ is a Borel set and it belongs to $\Pi^0_3$ class, as countable intersections of closed sets are closed, countable sums of closed sets are $\Sigma^0_2$ and countable intersections of $\Sigma^0_2$ sets are $\Pi^0_3$.
\subsubsection{$\Pi^0_3$-hardness of \textit{TpcN}}\label{lower_bound_tpcn}
We need to define the map $f\colon\mathbb{N}^\mathbb{N}\to[0,1]^\mathbb{N}$ that will be used to concatenate specific finite blocks constructed for each term in given natural sequence $M(n)$. We want to create a binary sequence $x$ using $M(n)$ such that for each $n\in\mathbb{N}$ at the position $B_n=\sum_{i=1}^n i$ in the constructed sequence the value of $\mathcal{L}_{B_n}(x)$ is equal to the Taylor-based approximation of $\log B_n$.
\par Let $(a_n)$ be a sequence of real numbers such that
\begin{equation*}
    t(n)=\sum_{k=0}^{M(n)}\frac{\log_2^{(k)}(a_n)}{k!}(B_n-a_n)^k
\end{equation*}
are proper partial sums for Taylor series of $\log_2B_n$ for each $n\in\mathbb{N}$. Set $T(n)=2^{\frac{1}{M(n)}}t(n)$. For $n\in\mathbb{N}$ append $\min\{n,[T(n)]\}$ of ones to a null sequence, where $[x]$ is integer part of $x$. If $n-[T(n)]>0$, then append also $n-[T(n)]$ zeroes. This creates a finite sequence $\alpha_{B_n}...\alpha_{B_{n+1}-1}$. Define $f\colon\mathbb{N}^\mathbb{N}\to\{0,1\}^{\mathbb{N}}$ to be the concatenation of the above transformation for each term in $M(n)\in\mathbb{N}^\mathbb{N}$. Of course both $\mathbb{N}^\mathbb{N}$ and $\{0,1\}^\mathbb{N}$ spaces are products of the discrete topology spaces. Thus we can deduce continuity of a map from the fact that for two initially equal sequences their images are initially equal. Observe that for $f$ this is the case, since both $B_n, a_n$ are uniquely defined. Therefore if two distinct $M_1(n), M_2(n)$ are initially equal, then the values of $t(n)$ are initially equal and appended finite blocks $\alpha_{B_n}...\alpha_{B_{n+1}-1}$ are initially the same.
\par Now observe that if we took a sequence from $P_3$, then our approximation becomes more and more accurate as $n$ tends to infinity. If a sequence is not from $P_3$, then the adjustment $2^{\frac{1}{{M(n)}}}$ prevents the image from being typical. Note that the argument requires $n$ to be arbitrarily large, since initial blocks can be nowhere near the desired value.
\begin{lemma}\label{tpcn}
The preimage $f^{-1}(\textit{TpcN})=P_3$.
\end{lemma}

\begin{proof}
Firstly observe that if we constructed a sequence $x$ using $M(n)\in P_3$, then, as $M(n)$ tends to infinity with respect to $n$, we conclude that for $\varepsilon>0$ there exists $n_0\in \mathbb{N}$ such that our approximation is $\varepsilon$-exact at the positions $B_{n}$ for $n\geq n_0$. The next point of $\varepsilon$-exactness occurs at $B_{n_0+1}$. With that in mind, consider some $n\in(B_{n_0},B_{n_0+1})\cap\mathbb{N}$. We want to check whether at the position $n$ in the constructed sequence $\mathcal{L}_{n}(x)$ is also converging to the desired ratio. Observe that the following inequalities hold:
\begin{equation*}
    \frac{\log_2 B_{n_0}-\varepsilon}{\log_2 B_{n_0+1}+\varepsilon}\leq \frac{L_n(x)}{\log_2 n}\leq \frac{\log_2 B_{n_0+1}+\varepsilon}{\log_2 B_{n_0}-\varepsilon}.
\end{equation*}
It is clear that $L_n(x)$ cannot be greater than its value at $B_{n_0}+1$ and cannot be smaller than its value at $B_{n_0}$. A similar reasoning gives the boundaries for $\log_2 n$. By the fact that $\lim_{n_0\to\infty}\frac{B_{n_0}}{B_{n_0+1}}=1$ we get $$\lim_{n\to\infty}\frac{L_n(x)}{\log_2 n}=1.$$
Now observe that, as $\frac{1}{x}\xrightarrow[x\to\infty]{} 0$, even a very poor approximation via Taylor series with a sequence $M(n)\in \mathbb{N}^{\mathbb{N}}\setminus P_3$ can yield an almost perfect result. Those approximations become indistinguishable for large $B_n$ because we are taking only the integer part of it. However, if $M(n)\in \mathbb{N}^{\mathbb{N}}\setminus P_3$, then $$\liminf_{n\to\infty} M(n)<\infty.$$ There exists an upper bound $N\in\mathbb{N}$ such that $M(n_k)< N$ for a bounded subsequence $M(n_k)$ of $M(n)$. Thus the amount of ones we append using Taylor-based transformation on that subsequence is bounded from below by $2^{\frac{1}{N}}\log_2(B_{n_k})$. We obtain that the value of $\mathcal{L}_n(y)$ for parts of $y$ constructed using $M(n_k)$ is bounded from below by $2^{\frac{1}{N}}>1$, since $N$ is finite. Thus $\mathcal{L}_n(y)$ cannot converge to 1. Consequently we have that $f(\mathbb{N}^{\mathbb{N}}\setminus P_3)\subset\textit{ExcN}=\{0,1\}^\mathbb{N}\setminus \textit{TpcN}$.
\end{proof}

\begin{corollary}
    The set $\textit{TpcN}$ is $\Pi_3^0$-complete.
\end{corollary}

\begin{proof}
    Observe that combining the reasoning from Section \ref{upper_bound_tpcn} and Lemma \ref{jest_Pi_3} with Theorem \ref{Wadge} and Lemma \ref{tpcn}, we obtain that $\textit{TpcN}\in\Pi_3^0$ and is $\Pi_3^0$-hard. Following the comment on class completeness from Section \ref{borel_hierarchy_introduced}, the set of typical numbers is $\Pi_3^0$-complete.
\end{proof}

\subsection{Proof of Theorem \ref{theoremB}}\label{preface}
Before proceeding with the proof, let us discuss the techniques we will use and the steps we will follow. We will start with showing that the Champernowne constant is typical. We will modify that number in the later stages to obtain numbers which are only typical or only normal. To achieve that, we will use Theorem \ref{Nakai}. It is clear that both $\textit{TpcN}\setminus\textit{Normal}$ and $\textit{Normal}\setminus\textit{TpcN}$ are Borel sets. To examine their Borel complexity, we will apply a method introduced in \cite{Bill}, constructing a reduction composed of two mutually independent actions.

\subsubsection{Typical and normal number}\label{both}
Of course, there are uncountably many numbers that are both typical and normal. Let us focus on the Champernowne constant, which in base 2 is of the form
\begin{equation*}
    x=0.\;1\;10\;11\;100\;101\;110\;111\;...
\end{equation*}
Observe that in this form there is one natural number with binary expansion of length 1, 2 of length 2, 4 of length 3 and so on. Each of those blocks ends with a sequence of ones. It turns out that $\mathcal{L}_{n}(x)$ computed at the last positions of those ends converge to 1. We will use the same reasoning as in Lemma \ref{tpcn} to show that this number is indeed typical.
\begin{lemma}
The Champernowne constant expressed in base 2 is typical.
\end{lemma}

\begin{proof}
Observe that at each position $p_n=\sum_{k=1}^n k 2^{k-1}=p_n=(n-1)2^{n}+1$ we have $L_{p_n}(x)=n$. Checking the condition (\ref{condition_above}), for sufficiently large $n\in\mathbb{N}$ we obtain the following:
\begin{equation*}
    \frac{L_{p_n}(x)}{\log_2 p_n}=\frac{n}{\log_2((n-1)2^{n}+1)}\approx\frac{n}{n+\log_2 n}\xrightarrow[n\to\infty]{}1.
\end{equation*}
Choose $\epsilon>0$ and let an arbitrarily large $n_0$ be such that for $p_{n}$ with $n\geq n_0$ the inequality $|L_{p_{n}}(x)-\log_2 p_{n}|<\varepsilon$ holds. Take a position $j\in(p_{n_0}, p_{n_0+1})\cap \mathbb{N}$. Similarly as in Lemma \ref{tpcn} we have the following:
\begin{equation*}
    \frac{n_0}{n_0+1}\approx\frac{\log_2 p_{n_0}-\varepsilon}{\log_2 p_{n_0+1}+\varepsilon}\leq \frac{L_j(x)}{\log_2 j}\leq \frac{\log_2 p_{n_0+1}+\varepsilon}{\log_2 p_{n_0}-\varepsilon}\approx\frac{n_0+1}{n_0}.
\end{equation*}
As $n_0\to\infty$ we obtain that $\lim_{j\to\infty} \frac{L_j(x)}{\log_2 j}=1$.

\end{proof}
\subsubsection{Strictly typical number}\label{typical only}
In the present section we want to reuse the calculations from Section \ref{both}, but bias the frequencies. As we know that the Champernowne constant can be expressed as
\begin{equation*}
    0.\;f_2(1)\;f_2(2)\;f_2(3)\;f_2(4)\;...,
\end{equation*}
where $f_2(n)$ is binary notation of $n$, let us introduce a function $g\colon\mathbb{N}\to 2^\mathbb{N}$ satisfying the following conditions (here $|f_2(n)|$ represents the length of binary notation of $n$ and $a^n$ is a finite block of length $n$ constructed by concatenating $n$ times the letter $a$):
\begin{equation*}
    g(n)=\left\{\begin{array}{ll}
         & 1^{|f_2(n)|},\;\textnormal{if there exists}\;k\in\mathbb{N}\;\textnormal{such that}\;n=2^k-1, \\
         & 0^{|f_2(n)|},\;\textnormal{otherwise}.
    \end{array}\right.
\end{equation*}
The number $y=0.\;g(1)\;g(2)\;g(3)...$ is of the following form:
\begin{equation*}
    y=0.\;1\;00\;11\;000\;000\;000\;111\;...
\end{equation*}
One can see that the computations from Section \ref{both} are valid also for $y$, but $y$ is clearly having frequencies for sequences of zeroes bigger than those for sequences for ones.
\subsubsection{Strictly normal number}\label{normal only}

Define $w(x)=\frac{1}{2^a}(x-1)+1$ for some $a\in\mathbb{N}_2$. Consider the number
\begin{equation*}
    z=0.\;[w(1)]_2\;[w(2)]_2\;[w(3)]_2...=0.\;\underbrace{1\;...\;1}_{2^a\;\textnormal{times}}\;\underbrace{10\;...\;10}_{2^a\;\textnormal{times}}\;\underbrace{11\;...\;11}_{2^a\;\textnormal{times}}\;...
\end{equation*}
This number is normal by Theorem \ref{Nakai}. Observe that as in the Champernowne constant, it has a pattern of sequences of ones, but of a different length and ending at a different position. This difference turns out to be crucial, resulting with $\mathcal{L}_n(x)$ converging to at least 2.
\begin{lemma}\label{normal_only_computations}
The number $z$ is exceptional.
\end{lemma}
\begin{proof}
Consider the position $p_n=\sum_{k=1}^n 2^a \,k 2^{k-1}=2^a((n-1)2^{n}+1)$. At this position $L_{p_n}(x)=n2^a$. Checking the condition (\ref{condition_above}), for sufficiently large $n\in\mathbb{N}$ we obtain that
\begin{equation*}
    \frac{L_{p_n}(x)}{\log_2 p_n}\approx\frac{n2^a}{(n+a)+\log_2 n}> \frac{n 2^a}{(n+\log_2 n)a}=\frac{n}{n+\log_2 n}\cdot\frac{2^a}{a}.
\end{equation*}
It is clear that depending on the choice of $a\in\mathbb{N}_2$ the limit ranges from 2 to $\infty$.
\end{proof}

\subsubsection{Subsets' hierarchy}\label{level4}
One can see that both sets $\textit{TpcN}\setminus\textit{Normal}$ and $\textit{Normal}\setminus\textit{TpcN}$ belong to $D_2(\Pi_3^0)$ class, as both $\textit{TpcN},\textit{Normal}$ are $\Pi_3^0$-complete. What remains is to find a suitable reduction from a $D_2(\Pi_3^0)$-complete set. To achieve that, we will follow the method introduced in \cite{Bill}. Namely, we will set two maps $f, g\colon \mathbb{N}^\mathbb{N}\to\{0,1\}^\mathbb{N}$ such that $f^{-1}(\textit{TpcN})=P_3=g^{-1}(\textit{Normal})$. Then, we will compose those reductions into one map $\phi$. The map $\phi$ will be a reduction from sets $C\setminus D$ and $D\setminus C$, where $C=\{a(n)\subset \mathbb{N}\colon a(2n)\to\infty\}, D=\{b(n)\subset \mathbb{N}\colon b(2n-1)\to\infty\}$, which are known to be $D_2(\Pi^0_3)$-complete, as stated in \cite{Bill}. Thus, it will follow that both $\textit{TpcN}\setminus\textit{Normal},\textit{Normal}\setminus\textit{TpcN}$ are $D_2(\Pi^0_3)$-complete sets.

\par Before we begin, let us state two crucial results from \cite{Ki}.
\begin{definition}
    Let $A$ be a subset of $\mathbb{N}$. We say that $A$ has density $d(A)$ if the limit $\lim_{n\to\infty}\frac{|A\cap\{1,...,n\}|}{n}=d(A)$.
\end{definition}
\begin{theorem}[\cite{Ki}, Theorem 5]\label{zero_density}
    Let $x\in\textit{Normal}$ and $x'$ be obtained by modifying the binary expansion of $x$ on a zero density subset of digits. Then $x'\in\textit{Normal}$.
\end{theorem}
\begin{theorem}[\cite{Ki}, Theorem 3]\label{g_reduction}
    There exists a reduction from $P_3$ to the set $D_0$ of all subsets of $\mathbb{N}$ with zero density.
\end{theorem}

 Let $g$ be the reduction from Theorem \ref{g_reduction} and $w=(w_i)_{i=1}^\infty$ be a binary sequence. Define $g'\colon \mathbb{N}^\mathbb{N}\to\{0,1\}^\mathbb{N}$ to act in the following way: to a given natural sequence $a_n$ associate a subset of $\mathbb{N}$ using the map $g$. Then modify the sequence $w$ on the positions $g(a_n)$ by changing the digits on those positions to 0. Thus, the number $w'=(w'_i)_{i=1}^\infty$, where
\begin{equation*}
    w'_i=\left\{
    \begin{split}
        &0,\;\text{if } i\in g(a_n),\\
        &w_i, \;\text{otherwise},
    \end{split}
    \right.
\end{equation*}
is the image $g'(a_n)$. One can see that if $w$ is normal, then the map $g'$ is a reduction from $P_3$ to $\textit{Normal}$, as it follows the reasoning from \cite{Ki}. We will denote by $g'(a_n)[w]$ the action of $g'(a_n)$ over the sequence $w$.

Let $B_n=2^n$ and $w=(w_i)_{i=1}^\infty$ be a binary sequence. For given natural sequence $a_n$ define $T(n)$ in the same way as in Section \ref{lower_bound_tpcn}. Construct a block $bl_n=(0,1^{l_n}, 0)$ by concatenating the sequence of consecutive ones of length $l_n={\min\{[T(n)], 2n\}}$ with a zero on each end. Define $f'\colon \mathbb{N}^\mathbb{N}\to \{0,1\}^\mathbb{N}$ to modify for each $n\in\mathbb{N}_5$ the sequence $w$ at the ends of initial segments $[w_1,...,w_{2^n}]$ by changing the last $|bl_n|$ digits in those segments to $bl_n$.
Thus, the number $w'=(w'_i)_{i=1}^\infty$, where $(w'_i)_{i=1}^{\infty}$ is defined by concatenation of consecutive blocks $[w'_1]=[w_1]$,
\begin{equation*}
[w'_{2^{n-1}+1},..., w'_{2^n}]=
\left\{
    \begin{split}
        &[w_{2^{n-1}+1},...,w_{2^n}],\;1\leq n\leq 4,\\
        &([w_{2^{n-1}+1},...,w_{2^n-|bl_n|}],bl_n),\;n\geq 5,
    \end{split} 
    \right.
\end{equation*}
is the image $f'(a_n)$. One can see that if $\limsup_{n\to\infty}\mathcal{L}_n(w)<1$, then the map $f'$ is a reduction from $P_3$ to $\textit{TpcN}$, as it follows the reasoning from Lemma \ref{tpcn}. We will denote by $f'(a_n)[w]$ the action of $f'(a_n)$ over the sequence $w$.

Let $w\in\textit{Normal}\setminus\textit{TpcN}$ be a number such that $\limsup \mathcal{L}_n(w)<1$. We will present one construction of such number, which directly follows from \cite[Section 5.1]{Madritsch}. For simplicity, denote by $\omega^l$ a concatenation of a block $\omega$ with itself $l$ times. For each $i\in\mathbb{N}$ let
\begin{equation*}
    w_i=(p_1^{\lceil i2^i\log i\rceil},...,p_{2^i}^{\lceil i2^i\log i\rceil})
\end{equation*}
be a block constructed as a concatenation of all blocks $p_j\in\{0,1\}^i$, where each block is raised to the same power $\lceil i2^i\log i\rceil$.

\begin{theorem}[\cite{Madritsch}, Section 5.1]\label{omega_konstrukcja}
    Let $(\omega_j)_{j=1}^\infty=(w_1^{l_1}, w_2^{l_2}, w_3^{l_3},...)$ be an infinite concatenation of blocks $w_i$, where for each block $w_i$ the blocks $p_i$ are arranged in an arbitrary order and each $w_i$ is raised to the power $l_i=i^{2^i}$. The number $\omega=(\omega_j)_{j=1}^\infty$ is normal.
\end{theorem}
Assume that for each $i\in\mathbb{N}$ the order of blocks $p_i$ is induced by the natural order of $\mathbb{N}$. Define $\omega$ to be a normal number from Theorem \ref{omega_konstrukcja} constructed with respect to that order.
\begin{proposition}\label{omega_jedynki}
    For each $i\in\mathbb{N}$, the sequence $1^{\lceil i2^i\log i\rceil}$ is the maximal sequence of consecutive ones in $w_i$. Moreover, the length of every other sequence of consecutive ones is bounded by $2i$.
\end{proposition}
\begin{proof}
    By the choice of the ordering of blocks $p_j$, we have that $p_{2^i-1}=1^{i-1}0$. Moreover, every block $w_i$ starts with 0. This implies that the sequence $1^{\lceil i2^i\log i\rceil}$ is separated by a zero on each end. On the other hand, observe that each block $p_j\neq 1^i$ contains at least one zero. Thus, besides the last block $1^{\lceil i2^i\log i\rceil}$, there is at least one zero in every $2i$ consecutive positions of $w_i$. The assertion follows immediately.
\end{proof}

\begin{lemma}\label{omega_nietypowe}
    The limit $\lim_{n\to\infty}\mathcal{L}_n(\omega)=\infty$.
\end{lemma}
\begin{proof}
    We know that at the end of each block $w_i$ there is a sequence of consecutive ones of length $i2^i\log i\leq |1^{\lceil i2^i\log i\rceil}|\leq i2^i(\log i+1)$. By Proposition \ref{omega_jedynki}, one can see that this sequence is maximal for all blocks $w_1^{l_1},...,w_i^{l_i}$. Thus we can evaluate the function $\mathcal{L}_n(\omega)$ at the ends of consecutive blocks $w_i^{l_i}$, $B_n=\sum_{i=1}^{n}|w_i^{l_i}|$, where $i^{2i+1}2^{2i}\log i\leq |w_i^{l_i}|\leq i^{2i+1}2^{2i}(\log i+1)$. Thus
    \begin{equation*}
        \frac{L_{B_n}(\omega)}{\log_2 B_n}\geq\frac{n2^n\log n}{\log_2(\sum_{i=1}^n i^{2i+1}2^{2i}(\log i+1))}>\frac{n2^n\log n}{\log_2 (2n)^{2n+3}}\xrightarrow[n\to\infty]{}\infty.
    \end{equation*}
    Moreover, for each $k\in[B_{n}, B_{n+1}]\cap\mathbb{N}$ we have that $$\frac{L_k(\omega')}{\log_2 k}>\frac{L_{B_n}(\omega')}{\log_2 B_{n+1}}\xrightarrow[n\to\infty]{}\infty,$$ which completes the proof.
\end{proof}
\begin{lemma}\label{omega'}
    Let $\omega'$ be obtained from $\omega$ by changing the maximal sequence of consecutive ones $1^{\lceil i2^i\log i\rceil}$ to a sequence $0^{\lceil i2^i\log i\rceil}$ in every $w_i$. Then $\omega'$ is normal.
\end{lemma}
\begin{proof}
    Remember that $i2^i\log i\leq |1^{\lceil i2^i\log i\rceil}|\leq i2^i(\log i+1)$ for each $w_i$. Moreover, in the construction of $\omega$, each block $w_i$ is repeated $i^{2^i}$ times. Thus we have that $A(n)=\sum_{i=1}^n |M_1(i)|\leq \sum_{i=1}^n i^{2i+1}2^i(\log i+1)<n^{2n+3}2^n$, where $|M_1(i)|$ is the total length of all maximal sequences of consecutive ones in $w_i^{l_i}$. Moreover, we have that $B(n)=\sum_{i=1}^n |w_i^{l_i}|\geq \sum_{i=1}^n i^{2i+1}2^{2i}\log i> (2n)^{2n}$. Thus
    \begin{equation*}
        \frac{A(n)}{B(n)}<\frac{n^{2n+3}2^n}{(2n)^{2n}}=\frac{n^3}{2^n}\xrightarrow[n\to\infty]{}0,
    \end{equation*}
    which implies that the subset of digits in the binary expansion of $\omega$ which realise the maximal sequences of consecutive ones in consecutive blocks $w_i^{l_i}$ is of zero density. Thus $\omega'$ is normal by Theorem \ref{zero_density}.
\end{proof}
\begin{lemma}
    The limit $\lim_{n\to\infty}\mathcal{L}_n(\omega')=0$.
\end{lemma}
\begin{proof}
    By Proposition \ref{omega_jedynki}, we know that, besides blocks $1^{\lceil i2^i\log i\rceil}$, the length of every sequence of consecutive ones in $w_i^{l_i}$ is bounded by $2i$. To obtain $\omega'$, we performed a modification of $\omega$ where we neutralised every block $1^{\lceil i2^i\log i\rceil}$ in $w_i$. Thus, evaluating the function $\mathcal{L}_n(\omega')$ at positions $B_n=\sum_{i=1}^n |w_i^{l_i}|$, we have
    \begin{equation*}
        \frac{L_{B_n}(\omega')}{\log_2 B_n}<\frac{2n}{\log_2(\sum_{i=1}^n (2i)^{2i})}<\frac{2n}{2n \log_2 2n}\xrightarrow[n\to\infty]{}0.
    \end{equation*}
    Moreover, for each $k\in[B_{n}, B_{n+1}]\cap\mathbb{N}$ we have that $$\frac{L_k(\omega)}{\log_2 k}<\frac{L_{B_{n+1}}(\omega)}{\log_2 B_{n}}\xrightarrow[n\to\infty]{}0,$$ which completes the proof.
\end{proof}

We are ready to construct the map $\phi\colon\mathbb{N}^\mathbb{N}\to\{0,1\}^\mathbb{N}$. Let $\omega'$ be defined as in Lemma \ref{omega'}. Set $$\phi(a_n)=f'(a_{2n-1})[g'(a_{2n})[\omega']],$$
meaning that we apply two modifications to the number $\omega'$: first using the action of $g'(a_{2n})$ and then $f'(a_{2n-1})$. It is clear that $\phi$ is continuous. We will show that it is the desired reduction.

\begin{proposition}
    Let $x$ be a binary sequence. If $x$ is normal, then the image $f'(a_n)[x]$ is normal.
\end{proposition}
\begin{proof}
    Observe that using $f'$ we modify at most $\sum_{i=5}^n 2i+2$ positions of the first $2^n$ digits. Since $\frac{2n^2+2n}{2^n}\to 0$, the modification happens on zero density subset of digits. The assertion follows immediately from Theorem \ref{zero_density}.
\end{proof}

\begin{lemma}
    The sequence $\phi(a_n)$ is normal if and only if $a_{2n}\in P_3$.
\end{lemma}
\begin{proof}
    Observe that if $a_{2n}\in P_3$, then the map $g'$ associates a zero density subset of $\mathbb{N}$ on which we modify the binary expansion of $\omega'$. Since $g'$ is based on a reduction $g$ from $P_3$ to $D_0$, whenever $b_{2n}\not\in P_3$, then $g(b_{2n})$ is a subset of positive density and thus changing the digits to 0 via $g'$ on a positive density subset of digits implies $g'(b_{2n})[\omega']\not\in\textit{Normal}$.
\end{proof}

\begin{lemma}
    The sequence $\phi(a_n)$ is typical if and only if $a_{2n-1}\in P_3$.
\end{lemma}
\begin{proof}
    Observe that the sequences of consecutive ones in the binary expansion of $g'(a_{2n})[\omega']$ grow too slow for $\mathcal{L}_k$ to converge to 1. We update the growth rate of the length of sequences of consecutive ones using the function $f'$. Now, following the computations from Lemma \ref{tpcn}, if $a_{2n-1}\in P_3$, then the we approximate $\log_2 B_n$ properly and $\mathcal{L}_k(\phi(a_n))\to 1$. If $a_{2n-1}\not\in P_3$, then the adjustment $2^{\frac{1}{a_{2n-1}}}$ causes $T(n)$ to repetitively create too long sequences of ones resulting with $\limsup\mathcal{L}_k(\phi(a_n))>1$.
\end{proof}

By the above assertions one can see that the map $\phi$ is a reduction from $C\setminus D$ to $\textit{Normal}\setminus\textit{TpcN}$ and from $D\setminus C$ to $\textit{TpcN}\setminus\textit{Normal}$.

\begin{corollary}
    The sets $\textit{TpcN}\setminus\textit{Normal}$, $\textit{Normal}\setminus\textit{TpcN}$ are $D_2(\Pi^0_3)$-complete.
\end{corollary}

\section*{Acknowledgements}
Author would like to thank the anonymous referee for their insightful remarks on the paper, as well as Dominik Kwietniak, Tomasz Kania, Florian Richter, Piotr Oprocha, Bill Mance and Dawid Bucki for helpful comments during author's work.

\end{document}